\documentclass{amsart}

\usepackage{amssymb,amsmath,amsthm,latexsym,amscd,amsfonts,mathrsfs,mathtools}

\usepackage[margin=3.35cm]{geometry}
\usepackage{ulem}
\usepackage{epsfig}
\usepackage{psfrag}
\usepackage{graphicx}
\usepackage{color}
\usepackage{hyperref}
\usepackage{cleveref}

\newtheorem{theorem}{Theorem}[section]
\newtheorem{conjecture}[theorem]{Conjecture}
\newtheorem{proposition}[theorem]{Proposition}
\newtheorem{corollary}[theorem]{Corollary}
\newtheorem{lemma}[theorem]{Lemma}
\newtheorem{question}[theorem]{Question}

\theoremstyle{definition}
\newtheorem{definition}[theorem]{Definition}

\newtheorem{example}[theorem]{Example}
\newtheorem{remark*}[theorem]{}

\theoremstyle{theorem}
\newtheorem{remark}[theorem]{Remark}

\DeclareMathOperator*{\dprime}{\prime \prime}

\newcommand{\rar}{\rightarrow}
\newcommand{\C}{\mathbb{C}}

\newcommand{\mz}{\mathcal{Z}}
\newcommand{\tr}{\mathrm{tr}}

\newcommand{\Z}{\mathbb Z}
\newcommand{\N}{\mathbb N}
\newcommand{\ot}{\otimes}

\newcommand{\Aut}{\text{Aut}}

\newcommand{\mg}{\mathcal{G}}

\def\scfig #1 #2 {\resizebox{#2}{!}{\includegraphics{#1}}}

\begin{document}
 \title[Bases and regular subfactors]{A few remarks on Pimsner-Popa bases
   and regular subfactors of depth 2}

 \author[K C Bakshi]{Keshab Chandra Bakshi}
 \address{Chennai Mathematical Institute, Chennai, INDIA}
 \email{bakshi209@gmail.com, kcbakshi@cmi.ac.in}

\thanks{The first named author was supported through a DST
  INSPIRE faculty grant (reference  no. DST/INSPIRE/04/2019/002754).}

\author[V P Gupta]{Ved Prakash Gupta} \address{School of Physical
  Sciences, Jawaharlal Nehru University, New Delhi, INDIA}
\email{vedgupta@mail.jnu.ac.in, ved.math@gmail.com}

\dedicatory{In memory of   Vaughan Jones, a true pioneer!}
\maketitle

\begin{abstract}
We prove that a finite index regular inclusion of
    $II_1$-factors with commutative first relative commutant is always
    a crossed product subfactor with respect to a minimal action of a
    biconnected weak Kac algebra. Prior to this, we prove that every
  finite index inclusion of $II_1$-factors which is of depth $2$ and
  has simple first relative commutant (respectively, is regular and
  has commutative or simple first relative commutant) admits a
  two-sided Pimsner-Popa basis (respectively, a unitary orthonormal
  basis).\\
\end{abstract}

 \section{Introduction}

Right from the early days of the evolution of the theory of operator
algebras, the methods of crossed product constructions and fixed point
subalgebras with respect to actions by various algebraic objects on
operator algebras have served extremely well to provide numerous
examples with specific properties as well as to be considered as
suitable candidates for structure results under certain given
hypotheses. One of the first such structure results (thanks to
Ocneanu, Jones, Sutherland, Popa, Kosaki and Hong) states that every
{\it irreducible regular} inclusion of factors of type $II_1$ with
finite Jones index is a {\it group subfactor} of the form $N \subset N
\rtimes G$, with respect to an {\it outer action} of a finite group $G$
on $N$. In particular, every such subfactor has {\it depth} $2$.
Further, it has also been established (in a series of papers by
Ocneanu, David, Szyma\'{n}ski and Nikshych-Vainerman) that every finite
index inclusion of type $II_1$ factors of depth $2$ is of the form $N \subset
N \rtimes H$, with respect to a {\it minimal action} of some {\it
  biconnected weak Hopf $C^*$-algebra} $H$ - see \cite{Ocn, L, Dav,
  Szy, NV1}. More recently, Popa,  Shlyakhtenko and Vaes,
in \cite{PSV}, {among various interesting results}, classified
regular subalgebras $B$ of the hyperfinite $II_1$-factor $R$
with $B'\cap R = \mz(B)$. However, they do not provide any
structure for non-irreducible regular inclusions of factors of type
$II_1$. This short note is a first naive attempt in this direction, in which
we prove the following:\smallskip \color{black}

\noindent {\bf \Cref{mainresult}}{ \it Let $N\subset M$ be a finite
  index regular inclusion of $II_1$-factors with commutative relative
  commutant $N^{\prime}\cap M$. Then, there exists a biconnected weak
  Kac algebra $K$ and  a minimal action  of $K$ on $N$ such that
  $N\subset M$ is isomorphic to $N\subset N\rtimes K$.}\smallskip 

 It must be mentioned here that Ceccherini-Silberstein (in \cite{S})
 claimed to have proved that every finite index regular
 subfactor is a crossed product subfactor with respect to an outer
 action of a finite dimensional Hopf $C^*$-algebra. However, his assertion is incorrect and there is
 an obvious oversight in his proof as is pointed out in \Cref{gap}.\smallskip

\Cref{mainresult} is achieved by first proving that any finite index
regular inclusion of $II_1$-factors with commutative first relative
commutant has depth $2$ and then an appropriate application of
Nikshych-Vainerman's characterization of depth $2$ subfactors yields
the desired structure. In order to take care of the first part, we
utilize the notion of unitary orthonormal basis by Popa to show (in
\Cref{cute}) that any regular subfactor with simple or commutative
relative commutant is of depth at most $2$.  It fits well
to mention here that, in fact, Popa had recently asked 
(in \cite{pop1}) whether every integer index irreducible inclusion of
$II_1$-factors admits a unitary orthonormal basis or not. It seems to
be a difficult question to answer in full generality. In fact, the
question can be asked for non-irreducible inclusions as well, and we
provide a partial answer in:\smallskip

\noindent {\bf \Cref{regularunitary}} {\it Let $N\subset M$ be a
  finite index regular inclusion of factors of type $II_1$. If
  $N^{\prime}\cap M$ is either commutative or simple, then $M$ admits a
  unitary orthonormal basis over $N$.  }\smallskip
 
Then, the second part of \Cref{mainresult} is taken care of by a
suitable application of the notion of two-sided basis for inclusions
of finite von Neumann algebras. In fact, somewhat related to Popa's
question, and equally fundamental in nature, is the question related
to the existence of a two-sided Pimsner-Popa basis for any {\it
  extremal} inclusion of $II_1$-factors, which was asked by Vaughan
Jones around a decade back at various places. This question too has
tasted too little success. In \cite{BG}, we had shown that every
finite index regular inclusion of $II_1$-factors admits a two-sided
Pimsner-Popa basis and we have suitably adopted the idea of
  its proof in proving \Cref{regularunitary}. We move one more step
closer towards answering Jones' question by proving the
following:\smallskip

\noindent{\bf \Cref{two-sided-depth-2}} {\it Let $N\subset M$ be a
  finite index inclusion of type $II_1$-factors of depth $2$ with
  simple relative commutant $N^{\prime}\cap M$. Then, $M_2$ admits a
  two-sided Pimsner-Popa basis over $M_1$.

Furthermore, $M$ also admits a two-sided basis over $N$.  }\smallskip

The flow of the article is in the reverse order in the sense that,
after some preliminaries in Section 2, we first make an attempt to
partially answer Jones' question regarding existence of two-sided
basis in the first half of Section 3 and then move towards Popa's
question regarding existence of unitary orthonormal basis in the
second half of the same section. Finally, in Section 4, we establish
that any regular subfactor with commutative first relative commutant
is given by crossed product by a weak Kac algebra.

 \section{Preliminaries}
Since there are slightly varying (though equivalent) definitions
available in literature, in order to avoid any possible confusion, we
quickly recall the definition that we shall be using here. For further
details, we refer the reader to \cite{BNS,NSW,NV1,NV2} and the
references therein.

\begin{definition}\cite{NSW,BNS}
  \begin{enumerate}
    \item A \textit{weak bialgebra} is a quintuple
      $(A,m,\eta,\Delta,\varepsilon)$ so that $(A,m,\eta)$ is an algebra,
      $(A,\Delta,\varepsilon)$ is a coalgebra and the tuple satisfies the
      following compatibility conditions between algebra and coalgebra
      structures:
  \begin{enumerate}
  \item $\Delta$ is an algebra homomorphism.
  \item
  $\varepsilon(xyz)=\varepsilon(xy_1)\varepsilon(y_2z)$
  and $\varepsilon(xyz)=\varepsilon(xy_2)\varepsilon(y_1z)$  for all $x,y, z\in A$.
  \item ${\Delta}^2(1)=\big(\Delta(1)\otimes 1\big)\big(1\otimes \Delta(1)\big)=\big(1\otimes \Delta(1)\big)\big(\Delta(1)\otimes 1\big).$
 \end{enumerate}

\item  A \textit{weak Hopf algebra} (or a \textit{quantum groupoid}) is
  a weak bialgebra $(A,m,\eta,\Delta,\varepsilon)$ along with a
  $k$-linear map $S:A\rightarrow A$, called an {\it antipode}, satisfying
  the following antipode axioms:
 \begin{enumerate}
  \item $x_1S(x_2)=\varepsilon(1_1x)1_2$,
  \item $S(x_1)x_2=1_1\varepsilon(x1_2)$ and
  \item $S(x_1)x_2 S(x_3)=S(x).$
 \end{enumerate}
\item A weak Hopf algebra $(A,m,\eta,\Delta,\varepsilon)$ is said to be
   a \textit{weak Hopf $C^*$-algebra} if $A$ is a finite dimensional
   $C^*$-algebra and the comultiplication map is $*$-preserving, i.e.,
   $\Delta(x^*)={\Delta(x)}^*.$
\end{enumerate}
  \end{definition}
 As in the preceding definition, throughout this paper, we shall use the
Sweedler's notation, i.e., $\Delta(x) = x_{(1)} \ot x_{(2)}$ and
$(\Delta \otimes \mathrm{Id})\Delta (x) = x_{(1)} \ot x_{(2)} \ot
x_{(3)} = ( \mathrm{Id}\otimes \Delta )\Delta (x)$ for all $x \in
A$.

\begin{definition}\cite{BNS}
 A \textit{weak Kac algebra} is a weak Hopf $C^*$-algebra
 $(A,m,\eta,\Delta,\varepsilon, S)$ such that $S^2=\mathrm{Id}_A$ and $S$
 is $*$-preserving.
\end{definition}
\begin{remark}
  \begin{enumerate}
    \item A weak Hopf algebra is a Hopf algebra if and only if the
 comultiplication is unit-preserving if and only if the counit is a
 homomorphism of algebras.

 In particular, every Kac algebra is a weak Kac algebra.
\item The dual of a weak Kac algebra also admits a canonical weak Kac algebra.
  \end{enumerate}
  \end{remark}

\begin{example} Given a finite groupoid $\mg$, 
the associated {groupoid algebra} $\C[\mg]$ inherits a canonical weak
Kac algebra structure with respect to the comultiplication $\Delta$,
the counit $\varepsilon$ and the antipode $S$ satisfying
\[
\Delta(g) = g \ot g, \varepsilon(g) = 1, S(g) = g^{-1}\ \text{ for}\  g \in \mg.
\] 
It is easily seen that $\C[\mg]$ (resp., $\C[\mg]^*$) is a
cocommutative (resp., commutative) weak Kac algebra. And, conversely,
it was proved by Yamanouchi that for every cocommutative weak Kac
algebra $H$ there exists a finite groupoid $\mg$ such that $H$ is
isomorphic to $\C[\mg]$.
\end{example}

 Given any weak Kac algebra $A$, the target (resp., source) counital
 map $\varepsilon^t$ (resp., $\varepsilon^s$) on $A$, is given by
 $\varepsilon^t(x)=\varepsilon(1_{(1)}x)1_{(2)}$ \big(resp.,
 $\varepsilon^s(x)=1_{(1)}\varepsilon(x1_{(2)})\big)$ for $x\in A$,
 where $\Delta(1) = 1_{(1)}\ot 1_{(2)}$ in Sweedler's notation.  These
 maps are idempotent, i.e., $\varepsilon^t \circ \varepsilon^t =
 \varepsilon^t$, $\varepsilon^s \circ \varepsilon^s = \varepsilon^s$,
 and their images are unital $C^*$-subalgebras (called the {\it
   Cartan subalgebras}) of $A$:
\[
 A_t:=\{x\in A: \varepsilon^t(x)=x\}\ \text{ and} \
 A_s:=\{x\in A: \varepsilon^s(x)=x\}.
 \]
$A$ is said to be connected if the
  inclusion $A_t\subset A$ is connected (see \cite{GHJ} for
  definition). And, $A$ is said to be \textit{biconnected} if both $A$
  and its dual are connected.

  \begin{remark}

 Given a finite groupoid $\mg$, the  groupoid algebra $\C[\mg]$ is biconnected if
  and only if $\mg$ is a group. 
    \end{remark}
\color{black}

\color{black}
\subsection{Crossed product construction}\( \)

We now briefly recall the notion of the crossed product construction via an action of 
a weak Hopf $C^*$-algebra, as in \cite{NSW} (also see \cite{NV2, NV1}).

\begin{definition}
  \begin{enumerate}
    \item By a (left) action of a weak Hopf $C^*$-algebra $A$ on a von
      Neumann algebra $M$, we mean a  linear
      map
      \[
      A\otimes M\ni a\otimes x \mapsto (a\vartriangleright x)\in
      M
      \]
      which defines a (left) module structure on $M$
      and satisfies the conditions
 \begin{enumerate}
  \item $a\vartriangleright xy=(a_{(1)}\vartriangleright x)(a_{(2)}\vartriangleright y)$,
  \item $(a\vartriangleright x)^*=S(a)^*\vartriangleright x^*,$ and
  \item $a\vartriangleright 1=\varepsilon^t(a)\vartriangleright 1$ and $a \vartriangleright 1=0$ iff $\varepsilon^t(a)=0$
  \end{enumerate}
 for $a \in A$, $x, y \in M$.
 
\item  Under such a (left) action, the \textit{crossed product algebra}
  $M\rtimes A$ is defined as follows:

  As a $\C$-vector space it is the relative tensor product
  $M\otimes_{A_t} A$, where $A$ (resp., $M$) admits a canonical left
  (resp., right) $A_t$-module structure so that \( x
  (z\vartriangleright 1)\otimes a \sim  x \otimes za, \) for all
  $x\in M, a\in A,z\in A_t.$ For each $(a, x) \in A \times M$,
  $[x\otimes a]$ denotes the class of the element $x\otimes a$ and a
  natural $*$-algebra structure on $M\otimes_{A_t} A$ is given by:
  \[
    [x\otimes a][y\otimes b]=[x(a_{(1)}\vartriangleright y)\otimes a_{(2)}b],
    \]
    \[
      [x\otimes a]^*=[(a_{(1)}^*\vartriangleright x^*)\otimes
        a_{(2)}^*],\] for all $x,y\in M$ and $a,b\in A$.

\item The action is said to be \textit{minimal} if $A^{\prime}\cap (M\rtimes
  A)=A_s.$
  \end{enumerate}
\end{definition}
  \begin{remark}\cite{NSW, NV1,NV2}
    \begin{enumerate}
\item $M\rtimes A$ can be realized as a von Neumann
      algebra.
 \item If $M$ is a $II_1$-factor and $A$ is a weak Hopf
$C^*$-algebra acting minimally on $M$, then $M \rtimes A$ is also a
$II_1$-factor.
\end{enumerate}
  \end{remark}
  \color{black}
  
Our interest in actions of weak Hopf $C^*$-algebras stems from the
following beautiful characterization of depth $2$ subfactors by
Nikshych and Vainerman.  Before stating them, it would be appropriate
to recall the following definition.
\begin{definition}\label{depth-defn}
 Consider a finite index inclusion $N\subset M$ of $II_1$-factors and
 suppose $N\subset M\subset M_1\subset \cdots \subset M_k\subset
 \cdots $ is its tower of Jones' basic construction. Then, the
 inclusion $N\subset M$ is said to have \textit{finite depth} if there
 exists a $k$ such that $N^{\prime}\cap M_{k-2}\subset N^{\prime}\cap
 M_{k-1}\subset N^{\prime}\cap M_{k}$ is an instance of basic
 construction.  The least such $k$ is defined as the \textit{depth} of
 the inclusion.
 \end{definition}\color{black}
We urge the reader to see \cite{GHJ} for various other equivalent
formulations of the notion of depth.\smallskip

For any finite index irreducible inclusion $N\subset M$ of $II_1$-factors, i.e.,
$N^{\prime}\cap M=\C$, it was announced by Ocneanu (in \cite{Ocn}) and
proved later, separately, by Szyma\'{n}ski, David and Longo  - see \cite{Szy,
  Dav, L} - that if $N \subset M$ is of depth $2$, then there exists a
 Kac algebra $K$ and a minimal action of $K$ on
$M_1$ such that $M_2\cong M_1\rtimes K$  and $M = M_1^H$. More generally,
Nikshych and Vainerman obtained the following characterization:
\begin{theorem}\cite{NV1, BNS}
 A finite index inclusion $N\subset M$ of $II_1$-factors is of depth
 $2$ if and only if there exists a biconnected weak Hopf
 $C^*$-algebra $H$ and a minimal action of $H$ on $M_1$ such that
 $M_2\cong M_1\rtimes H$ and
 $M = M_1^H$.
\end{theorem}

\color{black}

\section{Pimsner-Popa Bases}
Let $\mathcal{N}\subset \mathcal{M}$ be a unital inclusion of von
Neumann algebras equipped with a faithful normal conditional
expectation $\mathcal{E}$ from $\mathcal{M}$ onto $\mathcal{N}$. Then,
a finite set $\mathcal{B}:=\{\lambda_1,\ldots,\lambda_n\}\subset
\mathcal{M}$ is called a {\it left} (resp., {\it right) Pimsner-Popa
  basis} for $\mathcal{M}$ over $\mathcal{N}$ via $\mathcal{E}$ if
every $x\in \mathcal{M}$ can be expressed as $x=\sum_{i=1}^n
\mathcal{E}(x\lambda^*_i)\lambda_i$ (resp.,
$x=\sum_{j=1}^n\lambda_j\mathcal{E}(\lambda^*_jx$). Further, such a basis $\{\lambda_i\}$
is said to be {\it orthonormal} if $\mathcal{E}(\lambda_i \lambda_j^*)
= \delta_{i,j}$ for all $i, j$. And, a collection $\mathcal{B}$
is said to be a {\it two-sided basis} if it is simultaneously a left
and a right Pimsner-Popa basis.

In this article, when we do not use the adjectives left or
  right, by a basis we shall always mean a right Pimsner-Popa basis
  (and not a two-sided basis).

\subsection{Two-sided basis} \( \)

About a decade back, Vaughan Jones asked the following question at
various places\footnote{For instance, during the second talk by
M. Izumi in the workshop organized in honour of V S Sunder's 60th
birthday at IMSc, Chennai during March-April 2012.}.
\begin{question}(Vaughan Jones) \label{Jones}
 Let $N$ be a $II_1$-factor and $N\subset M$ be an extremal subfactor of
finite index. Then, does there always exist a two-sided Pimsner-Popa basis for
 $M$ over $N$?
\end{question}
      \begin{example}
  Given a finite group $G$ and a subgroup $H$, by Hall's Marriage
  Theorem, we can obtain a set of coset representatives which acts
  simultaneously as representatives of left and right
  cosets\footnote{
  \url{https://mathoverflow.net/questions/6647/do-subgroups-have-two-sided-bases.}}
  of $H$ in $G$. Therefore, if $G$ acts outerly on a $II_1$-factor
  $N$, then $N\rtimes G$ always possesses a two-sided unitary
  orthonormal basis over $N\rtimes H$.

  This observation, therefore, allows us to
  think about the existence of a two-sided basis as a subfactor
  analogue of Hall's Marriage Theorem.
        \end{example}

\begin{definition}    
An inclusion $\mathcal{Q} \subset \mathcal{P}$ of von Neumann algebras
is said to be regular if its group of normalizers
$\mathcal{N}_{\mathcal{P}}(\mathcal{Q}):=\{u\in
\mathcal{U}(\mathcal{P}):u\mathcal{Q}u^*=\mathcal{Q}\}$ generates
$\mathcal{P}$ as von Neumann algebra, i.e.,
$\mathcal{N}_{\mathcal{P}}(\mathcal{Q})'' = \mathcal{P}$.
\end{definition}

\begin{remark}
To the best of our knowledge, till date,  too little
progress has been made in answering \Cref{Jones}. 

\begin{enumerate}
    \item If $N\subset M$ is a regular irreducible subfactor of type
      $II_1$ of finite index, then (from some works of Ocneanu, Jones,
      Sutherland, Popa, Kosaki ) it is a well-known fact that
      it is isomorphic to $N\subset N\rtimes G$, for some outer action
      of a finite group $G$ on $N$ - see \cite{H}, for a precise
      statement. In particular, $M$ has a two-sided basis over $N$.

    \item In \cite{BG}, we could drop the irreducibility condition and
      showed, without depending upon any structure result, that every
      finite index regular subfactor $N \subset M$ of type $II_1$
      admits a two-sided basis. A little thought should
        convince the reader that the two-sided basis we constructed in
        \cite{BG} is in fact orthonormal.
\end{enumerate}
\end{remark}
A comment pertaining to an application of the notion of two-sided basis
fits in well here:
\begin{remark}
 It is a known fact to the experts that any regular subfactor of type
 $II_1$ has integer index - see \cite[Page 150]{GHJ}. However, there
 was no explicit proof easily accessible in literature until
 Ceccherini-Silberstein \cite{S} suggested having one. Though, the
   argument provided in \cite[Theorem 4.5]{S} seems incomplete as is
   indicated in \Cref{silmistake}. \smallskip

 To our satisfaction, we could do a little better (in \cite{BG}) by
 exhibiting that, for any finite index regular subfactor $N \subset M$
 of type $II_1$, its index is given explicitly by
 \begin{equation}\label{dimn-formula}
[M:N] =|G|\, \mathrm{dim}(N'\cap M),
\end{equation}
where $G$ denotes the generalized Weyl group of the inclusion $N
\subset M$, which is defined as the quotient group
$\frac{\mathcal{N}_M(N)}{\mathcal{U}(N) \mathcal{U}(N'\cap
  M)}$.\smallskip
\end{remark}
 Depending upon the structure result of irreducible depth
$2$ subfactors by Szyma\'{n}ski and a result by Kac which determines
when a Kac algebra is a group algebra, Nikshych and Vainermann deduced
(in \cite[Corollary 4.19]{NV1}) that a depth $2$ subfactor of type
$II_1$ with prime index $p$ is necessarily a group subfactor with
respect to an outer action of the cyclic group
$\Z/p\Z$. Interestingly, it turns out that the formula in
\Cref{dimn-formula} has the following analogous consequence.
\begin{proposition}
Let $N \subset M$ be a finite index regular inclusion of
$II_1$-factors. If $[M:N]=p$ is prime, then $N\subset M$ is irreducible.

In particular, the
cyclic group $G:=\Z/p\Z$ acts outerly on $N$ and $N \subset M$ is
isomorphic to $N \subset N \rtimes G$.
\end{proposition}
\begin{proof}
  Suppose, on contrary, that $N\subset M$ is not irreducible. Then,
  from \Cref{dimn-formula}, it follows that
  \[
[M:N]= \mathrm{dim}_{\C}(N'\cap M).
  \]
Note that, if $\Lambda$ denotes the inclusion matrix of the inclusion
$\C \subset N'\cap M$, then $\|\Lambda\|^2 = \mathrm{dim}_{\C}(N'\cap
M)$. In particular, $\|\Lambda\|^2 = [M:N]$, which then implies that
$\C \subset N'\cap M \subset N'\cap M_1$ is an instance of basic
construction - see \cite[Theorem
  4.6.3 (vii)]{GHJ}. Thus, $N'\cap M \cong M_n(\C)$ for some $n \geq
2$; so that $[M:N]= n^2$. This contradicts the hypothesis that $[M:N]$
is a prime number. Hence, $N \subset M$ must be irreducible.

The asserted structure of $N \subset M$ is then well-known.\end{proof}

 Further, employing appropriate two-sided bases for the
inclusions $N \subset N \vee (N'\cap M)$ and $N \vee (N'\cap M)
\subset M$, the following useful observation was proved explicitly in
the first two paragraphs of the proof of \cite[Theorem 3.12]{BG}. We
will be using it crucially in the proof of \Cref{mainresult} and shall not
repeat the details here.  

\begin{proposition}\cite{BG}\label{scalar-index}
Let $N \subset M$ be a finite index regular inclusion of $II_1$-factors. Then, 
 the Watatani index of the restriction of $\tr_M$ to $N'\cap M$ is a scalar.
\end{proposition}

Adding to the list, we shall provide, in the next
section, an yet another application of the notion of two-sided basis
for regular inclusions.\smallskip

\subsubsection{One more step towards Jones' question}\( \)

Note that, any irreducible regular factorial inclusion of type $II_1$, being
isomorphic to a crossed product subfactor by a group, must be of depth 2 (see \cite{GHJ} or 
\Cref{depth-defn} for definition).  Thus, it is natural to
ask the following question:

 \begin{question}
  Let $N\subset M$ be a depth $2$ subfactor of type $II_1$ of finite
  index. Then, does $M/N$ always have a two-sided basis?
 \end{question}

We do not know the answer yet in this generality. However, we provide
a partial answer in \Cref{two-sided-depth-2}, for which we require
some preparation.

First, we need (a mild generalization of) a useful result of Popa
\cite[$\S\,$1.1.5]{pop2}. Popa had proved it for any (left or right)
orthonormal basis and it is easy to see that it holds for any (left or
right) Pimsner-Popa basis as well. Recall that a commuting square $(D,
C, B, A)$ of von Neumann algebras is said to be non-degenerate if
\[
\overline{\mathrm{span}[CB]}^{\mathrm{S.O.T.}} = A = \overline{\mathrm{span}[BC]}^{\mathrm{S.O.T.}}.
\]

\begin{lemma}(Popa)\label{popacommsqrs}
 Let $\mathcal{M}$ be a finite von Neumann algebra with a faithful
 normal tracial state and
 $(\mathcal{N},\mathcal{K},\mathcal{L},\mathcal{M})$ be a
 non-degenerate commuting square of von Neumann subalgebras of
 $\mathcal{M}$. Then, any right basis for $\mathcal{K}/\mathcal{N}$ is also a
 right basis for $\mathcal{M}/\mathcal{L}.$
\end{lemma}
\begin{proof}
 Suppose $\{\lambda_i:i\in I\}$ is a right basis for
 $\mathcal{K}/\mathcal{N}$. Then, $\sum_i\lambda_i
 e^{\mathcal{K}}_{\mathcal{N}}\lambda^*_i=1$, where
 $e^{\mathcal{K}}_{\mathcal{N}}$ denotes the Jones projection corresponding
 to the inclusion $\mathcal{N}\subset \mathcal{K}$. Let $\Omega$ denote the canonical cyclic vector for $L^2(\mathcal{M})$. Then, for any $x\in
 \mathcal{L}$ and $y\in \mathcal{K}$, we have
 \begin{align*}
 \sum_i \lambda_i e^{\mathcal{M}}_{\mathcal{L}} \lambda^*_i(yx\Omega)
 &= \sum_i \lambda_i E^{\mathcal{M}}_{\mathcal{L}}(\lambda^*_iyx)\Omega\\
 &= \sum_i \lambda_i E^{\mathcal{M}}_{\mathcal{L}}(\lambda^*_iy)x\Omega\\
 &= \sum_i \lambda_i E^{\mathcal{K}}_{\mathcal{N}}(\lambda^*_iy)x\Omega~~~~~~~~~~~~~~~\qquad \text{[by commuting square condition]}\\
 &= yx\Omega.
 \end{align*}
As the commuting square is non-degenerate, we have
$\overline{\mathrm{span}\, \mathcal{L}\mathcal{K}}^{\mathrm{SOT}}=\mathcal{M}=
\overline{\mathrm{span}\, \mathcal{K}\mathcal{L}}^{\mathrm{SOT}}$. In
particular,
$\overline{[\mathrm{span}\, \mathcal{L}\mathcal{K}] \Omega}^{\|\cdot\|_2}=
L^2(\mathcal{M})=
\overline{[\mathrm{span}\, \mathcal{K}\mathcal{L}]\Omega }^{\|\cdot\|_2}$.
Therefore, we conclude that $\sum_i
\lambda_ie^{\mathcal{M}}_{\mathcal{L}}\lambda^*_i=1$ and the proof is
complete.
\end{proof}

Some specific conditions guarantee  non-degeneracy of some commuting squares.

\begin{lemma}\label{non-degenerate}
  Let $\mathcal{M}$ be a finite von Neumann algebra with a faithful
  normal tracial state and $(\mathcal{N}, \mathcal{P}, \mathcal{Q},
  \mathcal{M})$ be a commuting square consisting of von Neumann
  subalgebras of $\mathcal{M}$.  If, either
  \begin{enumerate}
\item $\mathcal{Q}\subset \mathcal{M}$ is an inclusion of
  $II_1$-factors with finite index and $\mathcal{N}\subset \mathcal{P}$ is a connected
  inclusion of finite dimensional $C^*$-algebras with
  \( [\mathcal{M}  : \mathcal{Q}] =\|\Lambda\|^2\),
  where $\Lambda$ denotes the
  inclusion matrix of $\mathcal{N}\subset \mathcal{P}$; or
  \item both $\mathcal{N}\subset \mathcal{P}$ and $\mathcal{Q}\subset
    \mathcal{M}$ are connected inclusions of finite dimensional
    $C^*$-algebras with \(  \|\Lambda\|^2=\|\Gamma\|^2\), where
     $\Lambda$  and  $\Gamma$ denote the respective inclusion matrices,
\end{enumerate}
  then $(\mathcal{N}, \mathcal{P}, \mathcal{Q}, \mathcal{M})$ is
  non-degenerate.  
\end{lemma}
\begin{proof}
A proof  can be obtained on similar lines as
that of \cite[Lemma 18]{BK}  based on the characterization of a basis illustrated in \cite[Theorem 2.2]{B}.  \end{proof}

\color{black} The next useful observation is a straight forward adaptation of
\cite[Proposition 3.3]{BG}, which uses the notion of path algebras
associated to inclusions of finite dimensional $C^*$-algebras by
Sunder and Ocneanu. We skip the details. 
\begin{lemma}\cite{BG} \label{path-algebra-result}
Let $A$ be a finite dimensional $C^*$-algebra and $\tr$ be a faithful
tracial state on $A$. Then, $A$ has a two-sided orthonormal basis over
$\C$ with respect to $\tr$.
\end{lemma}

The following interesting observation is a folklore.
\begin{proposition}\label{downward depth}
  Let $N\subset M$ be a finite index depth $2$ subfactor of type $II_1$ and $M
  \supset N \supset N_{-1}\supset N_{-2}\supset \cdots \supset N_{-k}
  \supset \cdots$ be a tunnel construction for $N \subset M$. Then,
  $N_{-2k}\subset N_{-2k +1}$ has depth $2$ for all $k \geq 1$.

Moreover, $M_{2k-1} \subset M_{2k}$ is also of depth $2$ for all $ k\geq 1$.
\end{proposition}

\begin{proof}
  For the tunnel part, it suffices to show that $N_{-2}\subset N_{-1}$
  has depth $2$.

  Let $\Gamma$ and $\Omega$ denote the inclusion matrices for the
inclusions $(N_{-2}'\cap N_{-1} \subset N_{-2}'\cap N )$ and
$(N_{-2}'\cap N \subset N_{-2}'\cap M )$, respectively. Then, by
\cite[Theorem 4.6.3]{GHJ}, it will follow that $N_{-2} \subset N_{-1}$
has depth $2$ if we can show that $\|\Gamma\|^2 < [N_{-1}: N_{-2}] =
\|\Omega\|^2 $.

Consider the Jones' basic construction tower
  \[
N_{-2} \subset N_{-1} \subset N \subset M \subset M_1 \subset M_2 \subset M_3 \cdots \subset M_k \subset \cdots.
\]
By \cite[Theorem 2.13]{Bis}, there exists a $*$-isomorphism (the shift
operator) $\varphi: N_{-2}' \cap M \rar N'\cap M_2$ such that
$\varphi(N_{-2}'\cap N) = N'\cap M_1 $ and $\varphi(N_{-2}'\cap
N_{-1}) = N'\cap M $.  Thus, the truncated towers \( [N_{-2}'\cap
  N_{-1} \subset N_{-2}'\cap N \subset N_{-2}'\cap M ] \) and \(
[N'\cap M \subset N'\cap M_1 \subset N'\cap M_2] \) are isomorphic.
In particular, if $\Lambda_i$ denotes the inclusion matrix for the
inclusion $(N'\cap M_i \subset N'\cap M_{i+1} )$, then $\Lambda_0 =
\Gamma$ and $\Lambda_1 = \Omega$; so, by \cite[Theorem 4.6.3]{GHJ}, we
obtain \( \|\Omega\|^2 = \|\Lambda_1\|^2 = [M:N]= [N_{-1}: N_{-2}] \)
and \( \|\Gamma\|^2 = \|\Lambda_0\|^2 < [M:N] = [N_{-1}:N_{-2}]  \).\smallskip

For the basic construction part, it suffices to show that $M_1\subset
M_2$ has depth $2$. The shift operator $\psi: N'\cap M_2 \rar
M_1'\cap M_{4}$ does the job as above.
\end{proof}

\begin{corollary}\label{dual-depth-2}
  Let $N\subset M$ be a finite index depth $2$ subfactor of type $II_1$. Then, $M_k
  \subset M_{k+1}$ also has depth $2$ for all $k \geq 0$.

  In particular, $N_{-k}\subset N_{-k +1}$ has depth $2$ for all $k
  \geq 1$, for any tunnel construction $M \supset N \supset
  N_{-1}\supset N_{-2}\supset \cdots \supset N_{-k} \supset \cdots$ of
  $N \subset M$.
  \end{corollary}
\begin{proof}
It suffices to show that $M \subset M_1$ is of depth $2$.

Fix a $2$-step downward basic construction $N_{-2}\subset N_{-1}
\subset N$ of $N \subset M$. Then, by the preceding proposition,
$N_{-2} \subset N_{-1}$ is also of depth $2$. So, by \cite{NV1}, there
exists a biconnected weak Hopf $C^*$-algebra $H$ with a minimal action
on $N$ such that $(N^H \subset N)\cong (N_{-1} \subset N)$. Thus,
$N_{-1} \subset N$ is also of depth $2$, by \cite{BNS} (also see
\cite[Section 8.1]{NV2}). Thus, by \Cref{downward depth} again,
$M\subset M_1$ is also of depth $2$.
  \end{proof}

We are now all set for the theorem of this subsection.

\begin{theorem}\label{two-sided-depth-2}
  Let $N\subset M$ be a finite index inclusion of type $II_1$-factors
  of depth $2$ with simple relative commutant $N^{\prime}\cap
  M$. Then, $M_2$ admits a two-sided orthonormal basis over $M_1$.

 Furthermore, $M$ also admits a two-sided orthonormal basis over $N$.
\end{theorem}
\begin{proof} 
  Although some of the arguments below are well-known (see \cite{pop2}),
we provide sufficient details for the sake of self-containment and
convenience of the reader. \medskip

  \noindent {\bf Step I:} Any (left/right) basis for $M^{\prime}\cap M_2$
over $M^{\prime}\cap M_1$ is also a (left/right) basis for $M_2$ over $M_1$.  \medskip

Note that, by \Cref{popacommsqrs}, it suffices to show that 
the quadruple
$$
\begin{matrix}
  M_1 &\subset & M_2 \cr \cup &\ &\cup\cr M^{\prime}\cap M_1 &\subset & M^{\prime}\cap M_2
\end{matrix}
$$
is  a non-degenerate commuting square.  Towards this direction, first, recall that the quadruple
\[
\mathcal{G}_1 := \begin{matrix} M &\subset & M_1 \cr
  \cup &\ &\cup\cr N^{\prime}\cap M &\subset & N^{\prime}\cap M_1
\end{matrix}
\]
is a commuting square - see, for instance, \cite[Proposition 4.2.7]{GHJ},
wherein the bottom inclusion is connected.

Let $\Lambda$ denote the inclusion matrix for the inclusion
$N^{\prime}\cap M\subset N^{\prime}\cap M_1$. Since $N\subset M$ is of
depth $2$, as was recalled in \Cref{downward depth}, we have
$[M_1:M]=[M:N]= {\lVert \Lambda\Vert}^2.$ Therefore, by
\Cref{non-degenerate}, the quadruple $\mathcal{G}_1$ is a
non-degenerate commuting square. Thus, its extension (as defined in
\cite[$\S 1.1.6$]{pop2}) is given by the quadruple
\[
\mathcal{G}_2:=\begin{matrix}
  M_1 &\subset & M_2 \cr \cup &\ &\cup\cr N^{\prime}\cap M_1 &\subset & N^{\prime}\cap M_2
\end{matrix};
\]
and, by the proposition in $\S 1.1.6$ of \cite{pop2}, $\mathcal{G}_2$
is a non-degenerate commuting square as well. On the other hand, note
that if $\Gamma$ denotes the inclusion matrix for $M'\cap M_1 \subset
M^{\prime}\cap M_2$, then since $M \subset M_1$ is also of depth $2$
(see \Cref{dual-depth-2}), we have $\|\Gamma\|^2 = [M_1:M]= [M:N]=\|\Lambda^T
\|^2$; so, the commuting square
\[
\mathcal{G}_3:=\begin{matrix} N^{\prime}\cap M_1 &\stackrel{\Lambda^T}{\subset} &
N^{\prime}\cap M_2 \cr \cup &\ &\cup\cr M^{\prime}\cap M_1 &\stackrel{\Gamma}{\subset} &
M^{\prime}\cap M_2
\end{matrix}
\]
is also non-degenerate,  by \Cref{non-degenerate}. In particular,
concatenating $\mathcal{G}_2$ and $\mathcal{G}_3$, we deduce from
\cite[$\S 1.1.5$]{pop2} that the quadruple
$$
\begin{matrix}
  M_1 &\subset & M_2 \cr \cup &\ &\cup\cr M^{\prime}\cap M_1 &\subset & M^{\prime}\cap M_2
\end{matrix}
$$
is  a non-degenerate commuting square.
\medskip

\noindent {\bf Step II:} $M^{\prime}\cap M_2$ has a two-sided orthonormal basis
over $M^{\prime}\cap M_1$. \medskip

We assert that $\big( M'\cap M_1 \subset M^{\prime}\cap M_2\big)$ is
isomorphic to $\Big( M'\cap M_1 \subset (M^{\prime}\cap M_1) \otimes Q
\Big)$ for some unital $*$-subalgebra $Q$ of $(M'\cap M_1)' \cap
(M'\cap M_3)$. Once this is established, we can then readily deduce
from \Cref{path-algebra-result} that $M^{\prime}\cap M_2$ has a
two-sided orthonormal basis over $M^{\prime}\cap M_1$.

 Since $N'\cap M \ni x \mapsto Jx^*J \in M'\cap M_1$ is an
 anti-isomorphism and $N'\cap M$ is simple, so is $M^{\prime}\cap
 M_1$. Again, since $M \subset M_1$ is also of depth $2$, the tower
 \[
M'\cap M_1 \subset M'\cap M_2 \subset M'\cap M_3
 \]
is an instance of basic construction. So, $M^{\prime}\cap M_3$ is also simple.
Thus, it follows from \cite[Lemma 2.2.2]{GHJ} that $(M^{\prime}\cap
M_1)^{\prime}\cap (M^{\prime}\cap M_3)$ is simple and that
\[
M^{\prime}\cap M_3 \cong (M^{\prime}\cap
M_1)\otimes \big[(M^{\prime}\cap M_1)^{\prime}\cap (M^{\prime}\cap
  M_3)\big].
\]
 Suppose that
$M^{\prime}\cap M_1\cong M_n(\C)$ and that $M^{\prime}\cap M_3\cong M_n(\C)\otimes
 M_k(\C).$ Denote the intermediate subalgebra corresponding to $M^{\prime}\cap M_2$ by
 $P$. It is well-known that $P$ is of the form $M_n(\C)\otimes Q$, where $Q$ is some unital $*$-subalgebra of $M_k(\C)$. We provide the details for the convenience of the reader. By \cite[Proposition 4.2.7]{GHJ} again, the quadruple
 \[
 \mathcal{G}_4:=\begin{matrix} P &\subset & M_n(\C)\otimes M_k(\C) \cr
 \cup &\ &\cup\cr {\big(M_n(\C) \ot 1\big)}^{\prime}\cap P &\subset &
 1\otimes M_k(\C).
 \end{matrix}
 \]
is also a commuting square. Note that, there exists a unital
$*$-subalgebra $Q$ of $M_k(\C)$ such that ${\big(M_n(\C) \ot
  1\big)}^{\prime}\cap P=1\otimes Q$. Clearly,  $M_n(\C)\otimes Q
\subseteq P$. To see the reverse inclusion, consider $x = \sum_i
a_i\otimes b_i \in P \subset M_n(\C)\otimes M_k(\C).$ Then, we have $x= \sum
(a_i\otimes 1) E_P(1\otimes b_i)$. Since $\mathcal{G}_4$ is a
commuting square, we immediately see that $E_P(1\otimes b_i)\in
1\otimes Q$ and hence $x\in M_n(\C)\otimes Q.$ In conclusion, we have
$P=M_n(\C)\otimes Q$, as was asserted. 

Thus, from Steps I and II, we deduce that  $M_2$  has a two-sided orthonormal basis over $M_1$.\medskip

Finally, fix any 2-step downward basic construction $N_{-2}\subset
N_{-1}\subset N\subset M$ for $N\subset M$. Then, by \Cref{downward
  depth}, $N_{-2}\subset N_{-1}$ also has depth $2$. Further, as seen
in \Cref{downward depth}, $N_{-2}'\cap N_{-1} \cong N'\cap M$ is
simple. Hence, we readily deduce from the preceding discussion that
$M$ must admit a two-sided orthonormal basis over $N$.
  \end{proof}
As an immediate consequence we deduce the following.
\begin{corollary}
 Every finite index irreducible inclusion of $II_1$ factors of depth 2 admits a two-sided orthonormal basis.
\end{corollary}

\begin{remark}
Note that a subfactor as in \Cref{two-sided-depth-2} need not be regular. For instance,
any  Kac algebra $K$,  which is not a group algebra, acts outerly on
the hyperfinite factor $R$ and yields a non-regular irreducible depth
$2$ subfactor.
  \end{remark}
\subsection{Unitary orthonormal basis} \( \)

We now move towards unitary orthonormal bases and touch upon another
fundamental question asked recently by Sorin Popa in \cite[$\S$
  3.5]{pop1}.
\begin{question}[Sorin Popa]\label{popa}
 Does there always exist an orthonormal basis consisting of $n$ many
 unitaries for an integer index ($=n$) irreducible inclusion of
 $II_1$-factors?
\end{question}

\begin{example}\label{subgroup-subfactor}
Let $H$ be a subgroup of a finite group $G$. If $G$ acts
  outerly on a $II_1$-factor $N$, then $N \rtimes G$ has a unitary
  orthonormal basis over $N\rtimes H$.
  In particular, every irreducible regular subfactor, being isomorphic to a group
  subfactor, admits a unitary orthonormal basis.
\end{example}
\color{black}

  In view of the preceding example, it is natural to ask whether we
  can drop the irreducibility condition or not. The following remarks fall in
  place here:
\begin{remark}\label{silmistake}
  \begin{enumerate}
\item The question of existence of unitary orthonormal basis for a
  general finite index regular subfactor which is not necessarily
  irreducible (thus modifying \Cref{popa}) was discussed by
  Ceccherini-Silberstein \cite{S}. In fact, he asserted (in \cite[Theorem
    4.5]{S}) that if $N\subset M$ is a regular subfactor with finite
  index, then $M/N$ has a unitary orthonormal basis. However, his
  proof depends on a technique of Popa (\cite[Theorem 2.3]{Pop3})
  which holds for Cartan subalgebras. Since Popa's proof depends
  crucially on maximal abelian-ness of the subalgebra, it is not clear
  whether it holds, more generally, for regular subalgebras or
  not. So, the proof of \cite[Theorem 4.5]{S} seems to be incomplete;
  although, the statement may still be true, which we rephrase in
  \Cref{regular-u-onb}.
    \item Furthermore, Ceccherini-Silberstein (in \cite[Theorem 4.7]{S}) had
      also asserted that if $M/N$ has a unitary orthonormal basis then
      the subfactor $M_1\subset M_2$ is of the form $M_1 \subset M_1
      \rtimes H$ and {hence has depth $2$ (see for instance
        \cite{NV1})}, which then implies that $N \subset M$ is also of
      depth $2$ - see \Cref{downward depth}.  However, this is well
      known to be incorrect as every group-subgroup subfactor
      $(R\times H)\subset (R\rtimes G)$ always has a unitary
      orthonormal basis (see \Cref{subgroup-subfactor}) whereas it is
      not necessarily of depth $2$.
      \item Though not directly related to the present discussion, the
        characterization of index 3 subfactors provided in Corollary
        3.19 of \cite{S} is known to be incorrect.
\end{enumerate}
\end{remark}
\color{black}

\begin{conjecture}\label{regular-u-onb}
Let $N\subset M$ be a finite index regular inclusion of factors of
type $II_1$.  Then, $M/N$ has a unitary orthonormal basis.
\end{conjecture}

As a partial progress in the resolution of this conjecture, we prove
the following:
\begin{theorem}\label{regularunitary}
Let $N\subset M$ be a finite index regular inclusion of factors of
type $II_1$. If $N^{\prime}\cap M$ is either commutative or simple, then
$M$ admits a unitary orthonormal basis over $N$.
\end{theorem}

We will need the following couple of results to achieve this. Recall
that, for a unital inclusion $B \subset A$ of finite dimensional
$C^*$-algebras with inclusion matrix $\Lambda$, a tracial state $\tr$
on $A$ is said to be a Markov trace for $B \subset A$ if
\[
\Lambda^t \Lambda  \bar{t} = \|\Lambda\|^2 \bar{t},
\]
where $\bar{t}$ denotes the trace vector of the tracial state
$\tr$. For more on Markov trace, see \cite{GHJ,JS}.\color{black}
\begin{lemma}\label{markov-u-onb}
Let $A:=\C \oplus \C \oplus
\cdots \oplus \C$ ($n$-copies).\begin{enumerate}
  \item
    The Markov trace  $\tr : A \rar \C$ for the unital inclusion $\C \subset A$ is given by
    \[
    \tr((z_1, \ldots, z_n)) = \frac{1}{n} \sum_i
    z_i.
    \]
\item  There exists a unitary orthonormal basis for $A$ over $\C$ with respect to the Markov trace
if and only if  there exists a unitary matrix $U=[u_{ij}]\in M_n(\C)$ such that $
|u_{ij}| = \frac{1}{\sqrt{n}}$ for all $1 \leq i, j \leq n$.
\end{enumerate}
\end{lemma}
 \begin{proof} 
   (1) This follows easily from \cite[Proposition 2.7.2]{GHJ}. 

   \smallskip
   
 \noindent (2):   $(\Rightarrow)$ Let $\{\lambda_i=(z_i^{(1)}, z_i^{(2)},
   \ldots, z_i^{(n)}) : 1 \leq i \leq n\}$ be a unitary orthonormal
   basis for $A$ over $\C$. Consider the matrix \( U =[u_{ij}] \in M_n
   \) whose entries are given by \( u_{ij}=
   \frac{1}{\sqrt{n}}z_{i}^{(j)} \), i.e., whose $i$-th column
   constitutes of the complex entries in $\lambda_i$. Then,
   $|u_{ij}|=\frac{1}{\sqrt{n}}$ for all $1 \leq i, j \leq n$ and
\[
U^*U= \begin{bmatrix}
\tr(\lambda_1^*\lambda_1) & \tr(\lambda_1^*\lambda_2) & \cdots & \tr(\lambda_1^*\lambda_n)\\
\tr(\lambda_2^*\lambda_1) & \tr(\lambda_2^*\lambda_2) & \cdots & \tr(\lambda_2^*\lambda_n)\\
\vdots & \vdots & \ddots & \vdots \\
\tr(\lambda_n^*\lambda_1) & \tr(\lambda_n^*\lambda_2) & \cdots & \tr(\lambda_n^*\lambda_n)\\
\end{bmatrix}=  I_n.
\]

\noindent $(\Leftarrow)$ Let $U=[u_{ij}]\in U(n)$ be such that $
|u_{ij}| = \frac{1}{\sqrt{n}}$ for all $1 \leq i, j \leq n$. Consider \[
\lambda_i := \sqrt{n}\, (u_{1i}, u_{2i}, \ldots, u_{ni}) \in A,\ i =1 , 2, \ldots, n.
\]
Since
$|u_{ij}| = \frac{1}{\sqrt{n}}$ for all $1 \leq i , j \leq n$, it
follows that $\lambda_i^*\lambda_i=(1, 1, \ldots, 1)$ for all $1 \leq
i \leq n$, i.e., each $\lambda_i$ is a unitary in $A$. 
Further, note that
\[
I_n= U^*U = [u_{ij}]^*[u_{ij}] =
\begin{bmatrix}
\tr(\lambda_1^*\lambda_1) & \tr(\lambda_1^*\lambda_2) & \cdots & \tr(\lambda_1^*\lambda_n)\\
\tr(\lambda_2^*\lambda_1) & \tr(\lambda_2^*\lambda_2) & \cdots & \tr(\lambda_2^*\lambda_n)\\
\vdots & \vdots & \ddots & \vdots \\
\tr(\lambda_n^*\lambda_1) & \tr(\lambda_n^*\lambda_2) & \cdots & \tr(\lambda_n^*\lambda_n)\\
\end{bmatrix}.
\]
Hence, $\tr(\lambda_i^*\lambda_j) = \delta_{i,j}$ for all $ 1 \leq i,
j \leq n$, which implies that $\{\lambda_1, \ldots, \lambda_n\}$ forms
a unitary orthonormal basis for $A$ over $\C$ with respect to above
tracial state.
 \end{proof}

 Recall that a {\it
  unitary error basis} for a matrix algebra $M_n(\C)$ is a Hamel basis
that is orthogonal with respect to the inner product induced by the
canonical trace of $M_n(\C)$. Little is known about their
structure. There are  two popular methods of construction
of unitary error bases. One is algebraic in nature (due to Knill) and the
other combinatorial (due to Werner). 
 
\begin{proposition}\label{unitary-onb}
 Let $A$ be a finite dimensional $C^*$-algebra which is either simple
 or commutative. Then, $A/\C$ has a unitary orthonormal basis with
 respect to the Markov trace for the unital inclusion $\C \subset A$.
\end{proposition}
\begin{proof}
 Suppose first that $A=M_n(\C)$ for some $n \geq 2$. The existence of
 a unitary orthonormal basis follows from the known construction of a
 unitary error basis. We include the
 details for the reader's convenience.

 We first recall such a basis for $n=2$ (because of its importance and
 popularity in quantum information theory).  The Pauli spin
 matrices (unitary error bases in dimension $2$) are defined as
 follows:
$$\sigma_x=\begin{bmatrix}
           0 & 1\\
           1 & 0
          \end{bmatrix},
          \sigma_y=\begin{bmatrix}
                    0 & -i\\
                    i & 0
                   \end{bmatrix},
                   \sigma_z=\begin{bmatrix}
                             1 & 0\\
                             0 & -1
                            \end{bmatrix}.
$$ It is an amazing fact that the set
                   $\{\mathrm{I}_2,\sigma_x,\sigma_y,\sigma_z\}$ forms
                   an orthonormal basis consisting of unitaries for
                   $M_2(\C).$
  
  For higher dimensions, consider the following two important matrices
  due to Sylvester and Weyl:
  \begin{equation*}
 U=
 \begin{bmatrix}
  1 & 0 & 0 & \cdots & 0\\
  0 & \omega & 0 &  \cdots & 0\\
  0 & 0 & {\omega}^2 & \cdots & 0\\
  \vdots & \vdots & \vdots & \ddots & \vdots\\
  0 & 0 & 0 & \cdots & {\omega}^{n-1}
 \end{bmatrix}
\end{equation*}
and
\begin{equation*}
V= 
\begin{bmatrix}
0 & 0 & 0 &   \cdots & 0  & 1\\
1 & 0 & 0 &   \cdots & 0  & 0 \\
0 & 1 & 0 &  \cdots  & 0  & 0\\
0 & 0 & 1 &  \cdots & 0 & 0\\
\vdots  & \vdots  & \vdots & \ddots & \vdots & \vdots\\
0 & 0 & 0 & \cdots & 1 &  0 
\end{bmatrix},
\end{equation*}
where $\omega := e^{-2\pi i /n}$ (a primitive root
   of unity). Then, it is known that the set $\{U^iV^j:
1\leq i,j\leq n\}$ forms a unitary error basis (in fact, a nice error
basis) for $M_n(\C)$.  These matrices also appeared in a work of Popa (\cite{pop}) (see also \cite{S}).  This
proves that $A/\C$ has unitary orthonormal basis whenever $A$ is
simple.

   Next, let $A$ be isomorphic to $\C\oplus \C \oplus \cdots \oplus
   \C$ ($n$-copies). Now, for a primitive root of unity $\omega$ as above,
   consider the well-known unitary DFT matrix
  \[
  U := \frac{1}{\sqrt{n}}
  \begin{bmatrix}
1 & 1 & 1 & 1& \cdots & 1\\
1 & \omega & \omega^2 & \omega^3 & \cdots & \omega^{n-1}\\
1& \omega^2 & \omega^4 & \omega^6 & \cdots & \omega^{2(n-1)}\\
1& \omega^3 & \omega^6 & \omega^9 & \cdots & \omega^{3(n-1)}\\
\vdots & \vdots & \vdots & \vdots& \ddots & \vdots\\
1& \omega^{n-1} & \omega^{2(n-1)} & \omega^{3(n-1)} & \cdots & \omega^{(-1)(n-1)}
    \end{bmatrix}.
  \]
  Clearly, each entry of $U$ has modulus $1/\sqrt{n}$. So,  by \cref{unitary-onb}, there exists a
  unitary orthonormal basis for $A/\C$ with respect to the Markov
  trace.
\end{proof}

The preceding observation will prove to be very crucial in the next
section. Thus, it seems it is worthwhile
to investigate in detail the existence of unitary basis of the finite
dimensional inclusions $B\subset A$, which, in turn, may prove to be
useful in answering the question of Popa for hyperfinite irreducible
subfactors.

\begin{lemma}\label{markov}
Let $N \subset M$ be a finite index regular subfactor of type
$II_1$. Then, ${\tr_M}_{|_{N'\cap M}}$ is the Markov trace for the
inclusion $\C \subset N'\cap M$.
  \end{lemma}

\begin{proof} 
 As pointed out in \Cref{scalar-index},  it can be extracted from the proof of
 \cite[Theorem 3.12]{BG} that the Watatani index
 (\cite{W}) $\text{Ind}(\tr_M)$ of $\tr_M$ is a scalar. Thus, it follows from
 \cite[Corollary 2.4.3]{W} and \cite[Proposition 3.2.3]{JS} that
 ${\tr_M}_{|_{N'\cap M}}$ is indeed the Markov trace for the inclusion
 $\C\subset N^{\prime}\cap M$.
\end{proof}  
\smallskip

\noindent{\it Proof of \Cref{regularunitary}:} Consider the
intermediate von Neumann subalgebra $\mathcal{R}:=N\vee
(N^{\prime}\cap M)$. Then, as in the proof of \cite[Lemma 3.4]{BG}, we
see that $(\C,N^{\prime}\cap M, N,\mathcal{R})$ is a non-degenerate
commuting square.  By \Cref{markov}, ${\tr_M}_{|_{N'\cap M}}$ is the Markov
trace for $\C \subset N'\cap M$; so, by \Cref{unitary-onb}, there
exists a unitary orthonormal basis, say, $\{u_i:i\in I\}$ for
$N^{\prime}\cap M$ over $\C$.  Then, by \Cref{popacommsqrs}, $\{ u_i :
i \in I\}$ is a unitary orthonormal basis for $\mathcal{R}/N$ as well.

On the other hand, since $N \subset M$ is regular, from
\cite[Proposition 3.7]{BG}, we know that $M/\mathcal{R}$ also has a
unitary orthonormal basis, say, $\{v_j:j\in J\}$. We assert that
$\{v_ju_i:i\in I, j\in J\}$ is a unitary orthonormal basis for
$M/N$. It is easy to see that $\{v_ju_i:i\in I~~\text{and}~~j\in J\}$
is a Pimsner-Popa basis for $M/N$. Also,
$$
E^M_N(u^*_i v_j^*v_ku_l)=E^{\mathcal{R}}_N\circ
E^M_{\mathcal{R}}(u^*_iv_j^*v_ku_l)=
E^{\mathcal{R}}_N\big(u^*_iE^M_{\mathcal{R}}(v^*_jv_k)u_l\big)=\delta_{j,k}\delta_{i,l}.
$$
Thus, $\{v_ju_i:i\in I~~\text{and}~~j\in J\}$ is a unitary
orthonormal basis for $M/N$.  \hfill $\Box$

\subsection{Two-sided basis versus unitary orthonormal basis}\( \)

Some preliminary observations suggest that the above questions of
Jones (Question \ref{Jones}) and Popa (Question \ref{popa}) may be
intimately interrelated in the case of integer index (extremal)
subfactors. Below, we illustrate some such connections.

The following fact is implicit in \cite{S}.
\begin{lemma}\cite{S}\label{S}
Let $N\subset M$ be a subfactor of finite index.  If $M/N$ has a  unitary
orthonormal basis, then $M_1/M$ admits a two-sided unitary
orthonormal basis.

In particular, $N \subset M$ is extremal.
\end{lemma}
\begin{proof}
This proof is extracted verbatim from \cite{S}. Suppose $\{\lambda_i:
1 \leq i \leq n\}$ is a unitary orthonormal basis for $M/N$. Thus,
$\sum_i \lambda_ie_1\lambda^*_i=1.$ Now, put
$$
v_k=\sum_{i=0}^{n-1} {\omega}^{ki} \lambda_ie_1\lambda^*_i, 0 \leq k
\leq n-1,
$$
where $\omega$ is an $n$-th root of unity. In
\cite[Proposition 3.24]{S}, it has been shown that $\{v_k\}$ is a
unitary orthonormal basis for $M_1/M$. Clearly this is two-sided.

Next, recall that $N\subset M$ is extremal if and only if $M\subset
M_1$ is extremal - see, for instance, \cite{pop2}. Since $M_1/M$ has a two-sided basis,
it is easily seen (see \cite{BG}) that it is extremal.
\end{proof}

\begin{proposition}
 Let $N \subset M$ be a finite index hyperfinite subfactor of type
 $II_1$ with finite depth. If $M/N$ has a unitary orthonormal basis,
 then it also has a two-sided unitary orthonormal basis.
\end{proposition}

\begin{proof}
By \Cref{S}, it follows that $M_1/M$, and hence, $M_2/M_1$ has a
two-sided unitary orthonormal basis.  It is known that the standard
invariants of the extremal subfactors $N \subset M$ and $M_1 \subset
M_2$ are isomorphic.  Thus, by Popa's classification result (see
\cite{pop2}), $N \subset M$ and $M_1 \subset M_2$, both being
hyperfinite, are isomorphic. Hence, $N \subset M$ has a two-sided
basis.  This completes the proof.
 \end{proof}

It will be good to know an answer of the following natural question.
 \begin{question}
  If $N\subset M$ is a finite depth integer index subfactor of type
  $II_1$, then is it true that $M/N$ has a unitary orthonormal basis if
  and only if $M/N$ has a two-sided Pimsner-Popa basis?
 \end{question}

 \begin{remark}
  Note that even if it can be shown that a finite index subfactor with a
  two-sided basis also admits a unitary orthonormal basis, then in
  view of \cite{BG}, it will follow that \Cref{regular-u-onb} holds
  true.
  \end{remark}

 \section{Regular subfactors and weak Kac algebras}
As recalled in the introduction, a finite index irreducible regular
inclusion of $II_1$-factors is always of the form $N \subset N \rtimes
G$ with respect to an outer action of a finite group $G$.  It is then
natural to ask what happens if we drop the irreducibility condition.

\begin{remark}\label{gap}

 Employing Szyma\'{n}ski's characterization of depth $2$ (irreducible)
 subfactors, Ceccherini-Silberstein (in \cite[Theorem 4.6]{S}) asserted that
 every finite index regular subfactor $N \subset M$ of type $II_1$ is
 of the form $N \subset N \rtimes H$ with respect to an outer action
 of a finite dimensional Hopf $*$-algebra $H$ on $N$. However, it had the
 following obvious oversight:\smallskip

 If his assertion  is true, then it will automatically force
 $N\subset M$ to be irreducible, whereas he has claimed to
 have characterized regular subfactors sans irreducibility.\medskip

In fact, Ceccherini-Silberstein's oversight stems from  an incomplete
  proof of an assertion made in \cite[Theorem 4.5]{S}, as explained
below:\smallskip

In the proof of \cite[Theorem 4.6]{S}, in view of \cite[Theorem
  4.5]{S}, a unitary orthonormal basis $\{\lambda_i\}$ is chosen
  for $M/N$ and then it is deduced that $$ N^{\prime}\cap
  M_1=\text{Alg}\{\lambda_ie_1\lambda^*_i:i\in I\}.
$$ Note that, the family $\{\lambda_ie_1\lambda^*_i\}$ consists of
  mutually orthogonal projections with $\sum_i
  \lambda_ie_1\lambda^*_i= 1$.  Thus, if $N\subset M$ is regular with
  finite index, then according to \cite[Theorem 4.6]{S},
  $N^{\prime}\cap M_1$ is always commutative. However, this is known to be
  untrue. For instance, taking an irreducible regular subfactor
  $K\subset L$ and putting $N=\C\otimes K$ and $M=M_n(\mathbb
  C)\otimes L$, it can be seen that $N\subset M$ is regular (see
  \Cref{tensor}) with integer index and $N^{\prime}\cap M \cong
  M_n(\mathbb C)$; so that $N'\cap M_1$ is not commutative.\smallskip
\end{remark}

So, the
question of characterizing (finite index) regular subfactors of type
$II_1$ is still unresolved.\color{black}

\begin{lemma}\label{tensor}
 Let $N\subset M$ be a regular inclusion of von Neumann
 algebras. Then, $\C \otimes N\subset M_n \otimes M$ is also regular.
\end{lemma}

\begin{proof}
Note that $\{ u \otimes v: u \in U(n), v \in \mathcal{N}_M(N)\}
\subseteq \mathcal{N}_{M_n \otimes M}(\C \otimes N)$. Thus,
\[
[*\text{-alg}\, U(n)] \otimes [*\text{-alg}\, \mathcal{N}_M(N) ] \subseteq
*\text{-alg}\, \mathcal{N}_{M_n \otimes M} (\C \otimes N).
\]

It is enough to show that
\[
\{ u \otimes z: u \in U(n), z \in   M\} 
\subset \mathcal{N}_{M_n \otimes M} (\C \otimes N)''.
\]
Let $z \in M$ and $u \in U(n)$. Then, there exists a net $\{x_i\}$ in
$*\text{-alg}\, \mathcal{N}_M(N)$ such that $x_i
\stackrel{\mathrm{WOT}}{\longrightarrow} z$. Thus, $\{u \otimes x_i\}$ is a net
in $ [*\text{-AC}\, U(n)] \otimes [*\text{-alg}\, \mathcal{N}_M(N)
]$, which is a $*$-subalgebra of $*\text{-alg}\, \mathcal{N}_{M_n
  \otimes M} (\C \otimes N)$. Also, $u \otimes x_i
\stackrel{\mathrm{WOT}}{\longrightarrow} u \otimes z$. Hence, $u
\otimes z \in \left(*\text{-alg}\, \mathcal{N}_{M_n \otimes M} (\C
\otimes N)\right)''$.
\end{proof}

\begin{theorem}\label{cute}
   Let $N\subset M$ be a finite index regular inclusion of type
   $II_1$-factors such that $N^{\prime}\cap M$ is either simple or
   commutative.  Then, $N\subset M$ is of depth at most $2$.
\end{theorem}

\begin{proof}
  Note that, by \Cref{regularunitary}, $M$ admits a unitary
  orthonormal basis over $N$. More precisely, taking $\mathcal{R}:=N
  \vee (N'\cap M)$, we saw that $\mathcal{R}/N$ admits a unitary
  orthonormal basis, say, $\{u_i:i\in I\}\subset
  \mathcal{U}(N^{\prime}\cap M)\subset \mathcal{N}_M(N)$;
  $M/\mathcal{R}$ admits a unitary orthonormal basis $\{v_j:j\in
  J\}\subset \mathcal{N}_M(N)$; and then, taking $w_{i,j}=v_ju_i$ we
  saw that $\{w_{i,j}:i\in I, j\in J\}$ is a unitary orthonormal basis
  for $M/N$ and $\{w_{i,j}: (i, j) \in I \times J\}\subset
  \mathcal{N}_M(N).$ In particular, we have $\sum_{i,j}
  w_{i,j}e_1w^*_{i,j}=1.$

 Now, note that for any unitary $u\in N$ we have
 $w^*_{i,j}uw_{i,j}=v_{i,j}$ for some unitary $v_{i,j}\in
 N$. Thus,
 \[
 u(w_{i,j}e_1w^*_{i,j})u^*=w_{i,j}v_{i,j}e_1v^*_{i,j}w^*_{i,j}=
 w_{i,j}e_1w^*_{i,j}.
 \]
 This implies that  $w_{i,j}e_1w^*_{i,j}\in N^{\prime}\cap
 M_1$ for all $(i,j)\in I\times J$. Further, we readily see that
 \[
 (w_{i,j}e_1w^*_{i,j})e_2(w_{i,j}e_1w^*_{i,j})=\tau
 w_{i,j}e_1w^*_{i,j} \ \forall\ (i, j) \in I \times J.
 \]
 Hence, $w_{i,j}e_1w^*_{i,j}\in (N^{\prime}\cap M_1)e_2
 (N^{\prime}\cap M_1)$ for every $(i,j)\in I\times J$. Now, since
 $1=\sum_{i,j} w_{i,j}e_1w^*_{i,j}$ and that $(N^{\prime}\cap
 M_1)e_2(N^{\prime}\cap M_1)$ is an ideal in $N^{\prime}\cap M_2$, it
 follows that that $(N^{\prime}\cap M_1)e_2 (N^{\prime}\cap
 M_1)=N^{\prime}\cap M_2.$ Thus, in view of \cite[Theorem 4.6.3]{GHJ}, $N \subset M$ has depth at most $2$.
\end{proof}

\begin{corollary}\label{abelianrelativecommutant}
  If $N\subset M$ is a finite index regular inclusion of  type $II_1$ factors
  such that $N^{\prime}\cap M$ is commutative, then it has depth $2$.
\end{corollary}
Few remarks are in order which tell that the converse of the above
result need  not be true.
\begin{remark}
  \begin{enumerate}
  \item
    A depth $2$ subfactor having commutative first relative commutant need
    not be regular. For example, consider a finite dimensional Hopf
    $C^*$-algebra (that is, a Kac algebra) $K$, which is not a group
    algebra, acting minimally on a type $II_1$ factor $N$. Then,
    $N\subset N\rtimes K$ is a depth 2 subfactor. Being irreducible,
    this subfactor is not regular.
 \smallskip
 
 \item Notice that a depth 2 regular subfactor $N\subset M$ may have a
   non-commutative first relative commutant.  As an example, one may
 look at  the subfactor illustrated in \Cref{gap}.

\end{enumerate}
\end{remark} 
\begin{theorem}\label{mainresult}
Let $N\subset M$ be a finite
  index regular inclusion of $II_1$-factors with commutative relative
  commutant $N^{\prime}\cap M$. Then, there exists a biconnected weak
  Kac algebra $K$ and  a minimal action  of $K$ on $N$ such that
  $N\subset M$ is isomorphic to $N\subset N\rtimes K$.
\end{theorem}
\begin{proof}
 By \Cref{cute} and \Cref{abelianrelativecommutant}, we observe that
 $N \subset M$ has depth $2$.  Choose a $2$-step downward basic
 construction $N_{-2}\subset N_{-1}\subset N\subset M$.  Then,
 $N_{-2}\subset N_{-1}$ is also of depth two - see
 \Cref{dual-depth-2}. Let $K:=N_{-1}'\cap M$. From \cite{NV1}, it will
 follow that $K$ admits a biconnected weak Kac algebra structure (with
 an appropriate action on $N$) provided the Watatani index of
 ${\tr_N}_{|_{N_{-1}'\cap N}}$ is a scalar.

By \Cref{scalar-index}, we know that
$\text{Ind}({\tr_M}_{|_{N^{\prime}\cap M}})$ is a scalar.  Let $J:
L^2(N) \rar L^2(N)$ denote the modular conjugation operator. Since
$N_{-1} \subset N \subset M$ is an instance of basic construction, the
map $B(L^2(N)) \ni x \mapsto JxJ \in B(L^2(N))$ is an anti-isomorphism
that maps $N_{-1}'$ onto $M$ and $\tr_M = \tr_{N_{-1}'}\circ [J(\cdot)
  J]$; so that \( {\tr_M}_{|_{N'\cap M}} = {\big({\tr_{N_{-1}'}}\circ
  [J (\cdot) J]\big)}_{|_{N'\cap M}}\). Also, $ 
J(N'\cap M) = N_{-1}'\cap N$; so that, $N_{-1}'\cap N$ is commutative and $(\C \subset N'\cap
M) \cong (\C \subset N_{-1}'\cap N)$. Further, since $N_{-1}\subset N$
is extremal (being of depth $2$), we have
\[
{\tr_N}_{|_{N_{-1}'\cap N }} = {\tr_{N_{-1}'}}_{|_{N_{-1}'\cap N }}.
\]
Thus, ${\tr_M}_{|_{N'\cap M}}$ and ${\tr_N}_{|_{N_{-1}'\cap N }}$ have
same trace vectors and hence
\[
 \mathrm{Ind}({\tr_N}_{|_{N_{-1}'\cap N}}) =  \text{Ind}({\tr_M}_{|_{N^{\prime}\cap M}}) ,
\]
which is a scalar. Thus, by \cite[Corollary 4.7 and Theorem
  4.17]{NV1}, $K$ admits a weak Kac algebra structure, which is also
biconnected, by \cite[Remark 5.8 (ii)]{NV1}. Further, by
\cite[Propositions 6.1 and 6.3, and Remark 6.4 (i)]{NV1}, $K$ acts
minimally on $N$ such that $N \subset M$ is isomorphic to $N \subset N
\rtimes K$. This completes the proof. 
\end{proof}
We end our discussion with a few well-known classes of reducible
regular subfactors.
 \begin{example}\label{last-examples}
  \begin{enumerate}
   \item If a finite group $G$ acts innerly on a $II_1$ factor $N$ in
     such a way that $M= N\rtimes G$ is a $II_1$ factor, then the
     inclusion $N\subset M$ is regular and $N^{\prime}\cap M$ is
     non-trivial. \footnote{
     \url{https://mathoverflow.net/questions/364547/action-of-a-finite-group-on-a-finite-factor}}

\item Suppose $N$ is a type $II_1$ factor. Then, the depth $1$
  subfactor $\C\otimes N\subset M_n(\C)\otimes N$ is an example of a
  regular subfactor with simple first relative commutant ($\cong
  M_n(\C)$).
\item Let $P$ be a $II_1$-factor and $\alpha \in \Aut(P)$. Consider
  the diagonal inclusion
  \[
  N:=\left\{\begin{pmatrix}
x & 0 \\ 0 & \alpha(x) 
  \end{pmatrix} : x \in P\right\} \subset M: = P \otimes M_2(\C),
  \]
  which is well known to be a subfactor of type $II_1$ with $[M:N]=
  4$. If $\alpha$ is an outer automorphism, then it is well known and can be easily seen that 
  \[
  N'\cap M =\{\mathrm{diag}(\lambda,\mu) :
  \lambda, \mu \in \C \}.
  \]
  Next, recall the Connes' outer conjugacy invariants $p_0(\alpha)\in
  \N$ and $\gamma (\alpha)\in \C$ given by $\{ n \in \Z: \alpha^n \in
  \mathrm{Inn}(P)\} = p_0(\alpha) \Z$ and $\alpha (u) = \gamma
  (\alpha) u$ for some $u \in \mathcal{U}(P)$ with
  $\alpha^{p_0(\alpha)}=\mathrm{Ad}_u$. Note that, if $p_0(\alpha) =
  2$ (in particular, $\alpha$ is outer) and $\gamma(\alpha) \neq 1$,
  then it is known that $N \subset M$ has depth $2$ - see \cite[$\S
    7$]{NV1}. We show that $N \subset M$ is regular as well, i.e.,
  ${\mathcal{N}_M(N)}^{\dprime}=M$.
  
Let $Q:= {\mathcal{N}_M(N)}^{\dprime}$ and fix a $1\neq u\in
\mathcal{U}(P)$ such that ${\alpha}^2=\mathrm{Ad}_u$.  Since
   \[
\{\mathrm{diag}(1,c): c \in \mathbb{T}\setminus \{1\} \} \subseteq
                {\mathcal{N}}_M(N)\setminus N,
  \]
  it is clear that $Q\neq N$. In view of the fact that $N\subset M$ is
  a maximal subfactor (see, for instance,
  \cite[Theorem 5.4]{TW}), it is sufficient to show that $Q$ is a
  factor. To this end, we show that $Q^{\prime}\cap Q\neq
  \{\mathrm{diag}(\lambda, \mu) : \lambda, \mu \in \C\}$; this will
  then prove that $Q$ is a factor as $Q^{\prime}\cap Q\subseteq
  N^{\prime}\cap M \cong \C\oplus \C$.

  Suppose, on the contrary, that $Q^{\prime}\cap Q =
  \{\mathrm{diag}(\lambda , \mu) : \lambda, \mu \in \C\}$. Then, we
  see that the projection $p:= \begin{pmatrix} 1& 0 \\ 0 & 0
  \end{pmatrix}$
  belongs (to $Q'\cap Q$ and hence) to $Q$; thus,
  \[
  \{\mathrm{diag}(x,y) : x,y \in P\}= pMp\oplus
  (1-p)M(1-p)=pNp\oplus (1-p)N(1-p)\subseteq Q.
  \]
We assert that the diagonal subalgebra $  D := \{\mathrm{diag}(x,y) : x,y \in P\}$ is a maximal von Neumann
subalgebra of $M$, i.e., $\{ D, a\}''=M$ for any $a \in M \setminus
D$. Assuming this assertion, if we consider $x:= \begin{pmatrix} 0 & 1 \\ u & 0
  \end{pmatrix}\in M$, then  $x\notin D$ and  a simple calculation
shows that $x \in \mathcal{N}_M(N)\subset Q$; so, by the maximality of
$D$ in $M$, it follows that $Q = M$ (a factor), which contradicts the
assumption that $Q'\cap Q \cong \C \oplus \C$. Hence, it just remains
to prove the maximality of $D$ in $M$. Let $y \in M\setminus
D$. Without loss of generality, we can assume that $y=\begin{pmatrix}
0 & w \\ z & 0 \end{pmatrix} $ with $(w,z) \neq (0,0)$. Futher, we can
assume that $w \neq 0$. Since $P$ is a $II_1$-factor, it is
algebraically simple; so, we have $PwP =P$. Thus, for each $0 \neq a
\in P$, we have $a= \sum_{i} x_i w y_i$ for a finite collection
$\{x_i, y_i: 1 \leq i \leq n\}$ in $P$. Thus,
\[
\begin{pmatrix} 0 & a \\ 0 & 0 \end{pmatrix} = \sum_i \begin{pmatrix} x_i & 0 \\ 0 & 0 \end{pmatrix}
\begin{pmatrix} 0 & w \\ z & 0 \end{pmatrix} \begin{pmatrix} 0 & 0 \\ 0 & y_i \end{pmatrix}\in \{D, y\}''.
\]
Likewise, $\begin{pmatrix} 0 & 0 \\ a & 0 \end{pmatrix} \in \{D,
y\}''$. Hence, $\begin{pmatrix} 0 & a \\ b & 0 \end{pmatrix} \in \{D,
y\}''$ for all $a, b \in P$, which then implies that $\{D, y\}''= M$,
i.e., $D$ is maximal in $M$.  \color{black}
  \end{enumerate}

 \end{example}
 
\color{black}

\subsection*{Acknowledgements}
The authors would like to thank Yongle Jiang for bringing
\Cref{last-examples}(3) to our notice; and, Vijay Kodiyalam, Sebastien
Palcoux and Leonid Vainerman for many fruitful exchanges. The authors
would also like to mark their note of appreciation to
T.~Ceccherini-Silberstein for acknowledging our concerns related to
the oversights in some of his proofs (from \cite{S}).

\end{document}